\documentclass[sn-mathphys]{sn-jnl}

\jyear{2021}%

\def\bea{\begin{eqnarray}}
\def\eea{\end{eqnarray}}
\def\nn{\nonumber}
\theoremstyle{thmstyleone}%
\newtheorem{theorem}{Theorem}
\theoremstyle{thmstyletwo}%
\newtheorem{lemma}{Lemma}
\newtheorem{assumption}{Assumption}
\theoremstyle{thmstylethree}%
\newtheorem{definition}{Definition}%

\begin{document}

\title[Parallel algorithms for maximizing monotone $OSS$ functions]{Parallel algorithms for maximizing monotone one-sided $\sigma$-smooth functions}


\author[1]{\fnm{Hongxiang} \sur{Zhang}}\email{zhanghx010@emails.bjut.edu.cn}

\author[2]{\fnm{Yukun} \sur{Cheng}}\email{ykcheng@amss.ac.cn}

\author[3]{\fnm{Chenchen} \sur{Wu}}\email{wu\_chenchen\_tjut@163.com}

\author[1]{\fnm{Dachuan} \sur{Xu}}\email{xudc@bjut.edu.cn}

\author[4]{\fnm{Dingzhu} \sur{Du}}\email{dzdu@utdallas.edu}

\affil[1]{\orgdiv{Beijing Institute for Scientific and Engineering Computing}, \orgname{Beijing University of Technology}, \orgaddress{\city{Beijing}, \postcode{100124},  \country{P.R. China}}}

\affil*[2]{\orgdiv{School of Business}, \orgname{Suzhou University of Science and Technology}, \orgaddress{\city{Suzhou}, \postcode{215009}, \country{P.R. China}}}

\affil[3]{\orgdiv{College of Science}, \orgname{Tianjin University of Technology}, \orgaddress{\city{Tianjin}, \postcode{300384}, \country{P.R. China}}}

\affil[4]{\orgdiv{Department of Computer Science}, \orgname{University of Texa, Dallas}, \orgaddress{\city{Dallas}, \postcode{75083}, \country{USA}}}

\abstract{In this paper, we study the problem of maximizing a monotone normalized one-sided $\sigma$-smooth ($OSS$ for short) function $F(x)$, subject to a convex polytope (no need to downward-closed \cite{GSS2021}). A function $F(x)$ is one-sided $\sigma$-smooth if $\frac{1}{2}u^{T}\nabla^{2}F(x)u\leq \sigma\cdot\frac{\|u\|_{1}}{\|x\|_{1}}u^{T}\nabla F(x)$, for all $x,u\geq 0, x\neq 0$. This problem was first introduced by Mehrdad et al. \cite{GSS2021} to characterize the multilinear extension of  some set functions.
Different with the serial algorithm with name Jump-Start Continuous Greedy Algorithm by Mehrdad et al. \cite{GSS2021}, we propose Jump-Start Parallel Greedy (JSPG for short) algorithm, the first parallel algorithm, for this problem. The approximation ratio of JSPG algorithm is proved to be $((1-e^{-\left(\frac{\alpha}{\alpha+1}\right)^{2\sigma}})-\epsilon)$ for any any number $\alpha\in(0,1]$ and $\epsilon>0$, which improves the approximation ratio of JSCG algorithm in \cite{GSS2021}. We also prove that our JSPG algorithm runs in $(O(\log n/\epsilon^{2}))$ adaptive rounds and consumes $O(n \log n/\epsilon^{2})$ queries, where the number of adaptive rounds and function evaluation queries approximately matches the known results for parallel submodular maximization. In addition, we study the stochastic version of maximizing monotone normalized $OSS$ function, in which the objective function $F(x)$ is defined as $F(x)=\mathbb{E}_{y\sim T}f(x,y)$. Here $f$ is a stochastic function with respect to the random variable $Y$, and $y$ is the realization of $Y$ drawn from a probability distribution $T$. For this stochastic version, we design Stochastic Parallel-Greedy (SPG) algorithm, which achieves a result of $F(x)\geq(1 -e^{-\left(\frac{\alpha}{\alpha+1}\right)^{2\sigma}}-\epsilon)OPT-O(\kappa^{1/2})$, with the same time complexity of JSPG algorithm. Here $\kappa=\frac{\max \{5\|\nabla F(x_{0})-d_{o}\|^{2}, 16\sigma^{2}+2L^{2}D^{2}\}}{(t+9)^{2/3}}$  is related to the preset parameters $\sigma, L, D$ and time $t$.}




\maketitle

\section{Introduction}
\label{sec:introduction}

This paper studies the problem of maximizing a monotone normlized one-sided $\sigma$-smooth ($OSS$) function subject to a convex polytope (no need to downwards-closed \cite{GSS2021}).
A function $F$ is one-sided $\sigma$-smooth if
$\frac{1}{2}u^{T}\nabla^{2}F(x)u\leq \sigma\cdot\frac{\|u\|_{1}}{\|x\|_{1}}u^{T}\nabla F(x)$,
for all $x,u\geq \textbf{0}, x\neq \textbf{0}$ (See Appendix for examples of $OSS$ functions).
The problem to maximize $OSS$ functions plays an important role in many fields, including machine learning \cite{mm2019,smv2017}, document aggregation \cite{zvm2013},  web search \cite{fs2006}, recommender systems \cite{dhx2006,jj1998}.
In this paper, we consider the deterministic and stochastic settings of the $OSS$ maximization problems given the basis of the polyhedron constraint, and design two parallel algorithms for them, respectively.
Formally the deterministic $OSS$ maximization problem is defined as follows:
\bea\max F(x)~~~~s.t.~~x\in[0,1]^{n}, x\in P,\eea
where $P$ is a convex polytope, and $F:[0,1]^{n}\rightarrow  \mathbb{R}_{+}$ is a monotone normalized $OSS$ function. For the stochastic $OSS$ problem, let $x\in X\subset[0,1]^{n}$ be an optimization variable and $Y$ be a random variable, which both determine the choice of a stochastic function $f:X\times Y\rightarrow \mathbb{R}$. The stochastic $OSS$ maximization problem is formally defined as
\bea
\max F(x):=\max \mathbb{E}_{y\sim T}\left[f(x,y)\right]~~~~
s.t.~~x\in[0,1]^{n}, x\in P,
\eea
where $P$ is a convex polytope and $\textbf{0}\in P$, $F$ is the expected value of the stochastic function $f$ with respect to the random variable $Y$, and $y$ is the realization of the random variable $Y$ drawn from a distribution $T$. $\mathbb{E}_{y\sim T}\left[f(x,y)\right]$ is a monotone normalized $OSS$ function.

The concept of $OSS$ was first introduced by Mehrdad et al. \cite{GSS2021} to describe the properties of multilinear extension \cite{gmr2000} of submodular set functions or diversity functions. Similar to Lipschitz smoothness \cite{alkb2017}, the property of $OSS$ can control the approximation ratio and the complexity of related algorithms.
The main method to maximize the $OSS$ problem is the continuous greedy, whose core is to maximize the multilinear extension function. Submodularity ensures some nice properties for the multilinear extension. For instance, concavity along a direction $d\geq \textbf{0}$ is used to bound a Taylor series expansion in the continuous greedy analysis. Since nonsubmodular multilinear extensions will not have this concavity property, Mehrdad et al proposed a "one-sided $\sigma$-smoothness" condition which guarantees an alternative bound based on Taylor series. Many well studied functions such as continuous DR-submodular functions are $OSS$ functions when $\sigma=0$.
Chandra et al. \cite{cq2019, cvz2014} stated that the multilinear extension of submodular set function is concave in the non-negative direction. They gave a $(1-1/e-\epsilon)$-approximation for the maximization problem of submodular set function subject to a cardinality constraint.
When $\sigma> 0$, $OSS$ function includes the multlinear extension of a diversity function, proved by Mehrdad et al. \cite{GSS2021}. They provided a tight ($1-1/e^{(1-\alpha)(\alpha/(\alpha+1))^{2\sigma}} $)-approximation for the  maximization problem of monotone normalized $OSS$ function, but their polyhedral constraint must have a downward closed property.

Mehrdad et al. \cite{GSS2021} first adopted the Frank-Wolfe algorithm to solve $OSS$ problems. Our parallel algorithm is inspired by
the Frank-Wolfe Algorithm, but is different from the parallelism of the traditional discrete Greedy algorithms. To be specific, our parallel algorithm improves the computational efficiency by avoiding solving the optimization problem $\max_{\nu\in P}\nu^{T}\nabla F(x)$ in each iteration.
Recently, many researchers focused on the study of parallel algorithms~\cite{an2020,pjgb2019} to solve optimization problems under big data. 
Eric and Singer \cite{es2018} first proposed a parallelism for submodular set function maximization problem with cardinality constraints. For this problem, Eric et al. and Alina et al. \cite{brs2019,en2019} respectively designed a near-optimal $(1-1/e-\epsilon)$-approximation algorithm with $O(\log n/\epsilon^{2})$ rounds of adaptivity. In \cite{cq2019}, Chandra and Quanrud devised an adaptive algorithm to approximately maximize the multilinear relaxation of a submodular set function subject to packing constraints. The algorithm in \cite{cq2019} achieves a near-optimal $(1-1/e-\epsilon)$-approximation in
$O(\log^{2}m~\log n/\epsilon^{4})$ rounds, where $n$ is the cardinality of the ground set and $m$ is the number of packing constraints.
All the above algorithms are parallel versions of traditional greedy algorithms and respective variants for the submodular function maximization problems. However, there is no specific parallel algorithm to solve the $OSS$ problems.

Generally, one key step to optimize $OSS$ problems \cite{smv2017} or continuous submodular problems \cite{bls2020,mns2011} is to compute a linear optimization problem: $\max \nu^{T}\nabla F(x)$, where $x,~\nu$ belongs to a bounded constraint. However, solving this linear optimization problem requires the exact value of gradient $\nabla F(x)$, which is not feasible for stochastic $OSS$ problem. In this paper,  we focus on the case where $F$ is a monotone normalized $OSS$ function.
For the stochastic $OSS$ problem, an approach to avoid computing $\nabla F(x)$ is modifying the existing algorithms by replacing gradients $\nabla F(x)$
with their stochastic estimation $\nabla f(x,y)$. However, this modification may lead to arbitrarily
poor solutions \cite{hsk2017}. Recently,  Mokhtari  et al. \cite{mhk2020} confirmed that if the feasible region is uniformly bounded, the gradient function $\nabla F(x)$ is $L$-\emph{Lipschitz} continuous and the variance of the unbiased stochastic gradients $\nabla f(x,y)$ is bounded by the constant $\theta^{2}$, then $\mathbb{E}[\|\nabla F(x_{t})-d_{t}\|^{2}]\leq \kappa$, where $d_{t}$ is the Linear replacement of stochastic gradient $\nabla f(x,y)$, ($d_{t}=(1-\rho_{t})d_{t-1}+\rho_{t}\nabla f(x_{t}, y_{t})$, where $\rho_{t}=\frac{4}{t+8}^{2/3}$ is a positive step size dependent on time $t$ and the initial vector $d_{0}$ is $\mathbf{0}$) and $\kappa=\frac{\max \{5\|\nabla F(x_{0})-d_{o}\|^{2}, 16\sigma^{2}+2L^{2}D^{2}\}}{(t+9)^{2/3}}$ ($L, D$ are constants). This conclusion is very helpful for our design of parallel algorithms for stochastic $OSS$ problems.


\subsection{Main Contributions}

 In this paper, we propose two parallel algorithms to maximize the deterministic $OSS$ problem and the stochastic $OSS$ problem, subject to a convex polytope constraint (no need to downwards-closed), respectively.
\begin{itemize}
	\item For the deterministic monotone normalized $OSS$ maximization problem, we design the Jump-Start Parallel-Greedy (JSPG) algorithm by combining different techniques, including Jump-Start, Frank-Wolfe, Continuous Greedy and parallel computing.
JSPG algorithm iterates by constantly accessing the relationship between the $OPT$ and the gradient function $\nabla F(x)$, and finally outputs the solution. We theoretically prove that JSPG algorithm has an approximation ratio of $((1-e^{-\left(\frac{\alpha}{\alpha+1}\right)^{2\sigma}})-\epsilon)$ for any number $\alpha\in[0,1]$ and $\epsilon>0$, and it runs in $(O(\log n/\epsilon^{2}))$ adaptive rounds and consumes $O(n \log n/\epsilon^{2})$ queries. Furthermore, JSPG algorithm can be applied to compute the continuous submodular function, which can output
	a ($1-1/e-\epsilon$) approximately optimal solution.
	\item For the stochastic monotone normalized $OSS$ maximization problem, we design the stochastic Parallel-Greedy (SPG) algorithm, by combining different techniques, including the Linear replacement of stochastic gradient, Frank-Wolfe, Continuous Greedy and parallel computing. SPG algorithm outputs the final solution through constantly accessing the relationship between the $OPT$ and $d_{t}$ in each iteration. Furthermore, SPG algorithm is proved that the final output solution $F(x)$ satisfies   $F(x)\geq(1 -e^{-\left(\frac{\alpha}{\alpha+1}\right)^{2\sigma}}-\epsilon)OPT-O(\kappa^{1/2})$,  and its time complexity is the same with the one of JSPG algorithm.
\end{itemize}

\subsection{Organization}

The remainder of this paper is organized as follows. Section 2 introduced some definitions and necessary lemmas for the algorithms design. In Section 3 and 4, two parallel algorithms are proposed for the deterministic and stochastic $OSS$ problem, respectively. The theoretical analysis for algorithms' efficiency are also provided. The last section concludes this work.

\section{Preliminaries}\label{sec2}
In this section, we would introduce some definitions and notations in advance, which are used throughout the whole paper.
Let $F$ be a monotone normalized $OSS$ function. For any vectors $x,y\in[0,1]^{N}$, we say $x\leq y$, if and only if $x_{i}\leq y_{i}$ holds. Let $x\vee y$ be the coordinate-wise maximum of $x$ and $y$, and $x\wedge y$ be the coordinate-wise minimum.

\begin{definition}\label{def1}
{\bf $OSS$ function:} Given a continuously twice differentiable function $F: \mathbb{R}^{n}_{\geq 0} \rightarrow \mathbb{R}$. $F$ is defined to be is
one-sided $\sigma$-smooth ($OSS$ for short) at point $x\geq 0$, if $F$ satisfies
\bea
u^{T}\nabla^{2}F(x)u\leq \sigma\cdot\frac{2\|u\|_{1}}{\|x\|_{1}}u^{T}\nabla F(x),\label{eqde1}
\eea
for all $u\geq0$.  Function $F$ is $OSS$, if (\ref{eqde1}) holds at all points of its domain.
\end{definition}

By Definition \ref{def1}, we have that a $OSS$ function $F$ is $OSS$ at any non-zero point of its domain. This property of $OSS$ captures the concavity in the forward direction when $\sigma=0$. Furthermore, inequality (\ref{eqde1}) means that the second derivative of $F(x)$ is bounded by the linear form of its gradient function. In fact, this inequality is one of the necessary conditions for the approximation algorithm design for the non-convex maximization problem.


\begin{definition}\label{def2}
{\bf Monotone $\&$ Normalized:} A $OSS$ function $F: \mathbb{R}^{n}_{\geq 0} \rightarrow R$ is monotone if  $F(x)\leq F(y)$, for all $x,y\in \mathbb{R}^{n}_{\geq 0}$, and $x\leq y$; and $F(x)$ is normalized if it satisfies $F(\textbf{0})=0$.
\end{definition}

Generally, when a continuous Frank-Wolfe technique is used to solve $OSS$ problem, a local condition, i.e., $\eta$-local, can help to analyze the approximation ratio and the complexity of the algorithm.
\begin{definition}\label{def4}
{\bf $\eta$-local:} For any $\eta\geq 0$, $\epsilon\in[0,1]$, $x+\epsilon u\in P$, we say $F\in \mathbb{C}^{2}$ is $\eta$-local at $x,u$ if
\bea
u^{T}\nabla F(x+\epsilon u)\geq (1-\eta\epsilon)u^{T} \nabla F(x)
\eea
holds for all $u, x\in P$, where $\mathbb{C}^{2}$ is the twice continuous differentiable space.
\end{definition}

The core of a parallel algorithm is the technique of "threshold loop call". Generally, a proper threshold can be set by using an estimation of the optimal value $OPT$. Following, we propose some trivial bounds of $OPT$ before our parallel algorithms design.

\begin{lemma}\label{lem1}{\bf OPT BOUNDs:}
Let $OPT$ be the optimal value of the problem to maximize the monotone normalized $OSS$ function $F(x)$, subject to a convex polytope $P$.\footnote{The basis of the polyhedron $P$ is required to exist and to be solvable in polynomial time.} Denote $r$ to be the rank of the polytope $P$. Then
\bea
\max\{F(x)\mid\|x\|_{1}=r, x\in P\cap[0,1]^{n}\}\leq OPT
\leq\min\{ F(\max_{x\in P}\|x\|_{1}), F(\textbf{1})\},\label{bound1}\nn
\eea
or
\bea
F\left(r\frac{\{\arg \max_{z} \|z\|_{1}, z\in P\}}{\|z\|_{2}}\right)
\leq OPT 
\leq\min\{F(\max_{x\in P}\|x\|_{1}), F(\textbf{1})\}.\label{bound2}\nn
\eea

\end{lemma}

The lower bound in (\ref{bound1}) is obvious. As $r$ is the rank of polytope $P$, we have $x=r\frac{\{\arg \max_{z} \|z\|_{1}, z\in P\}}{\|z\|_{1}}\in P\cap [0,1]^n$. Thus $OPT\geq F(x)$, leading to the lower bound of $OPT$ in (\ref{bound2}). In addition, the monotone property of $F(x)$ promises the upper bound, since $x\leq \textbf{1}$.

The notations used in this paper are lised in Table \ref{tab:1}.

\begin{table}[h!]
	\caption{Symbol Description}
	\label{tab:1}
	\begin{tabular}{ll}
		\hline\noalign{\smallskip}
		\noalign{\smallskip}\hline\noalign{\smallskip}
		$\delta, \sigma, \epsilon, \eta $ &:~The preset parameters which are positive; \\
		\noalign{\smallskip}
		$\lambda$                         &:~The estimation value of $OPT$; \\
		\noalign{\smallskip}
		$P$  &:~Polyhedron: such as linear function polyhedron $Ax\leq b$, matrix\\
             &~~ polyhedron, etc.; \\
		\noalign{\smallskip}
		$\{e_{P}\}_{r}$ &:~The set of basis vectors of polytope $P$, where
                        $r$ is the rank of \\
                        &~~polytope $P$. Generally, the base of a polytope is the linear maximum\\
                        &~~independent group of the polytope, and the rank of a polyhedron \\
                        &~~is the number of elements in its linear maximum independent group;\\
		\noalign{\smallskip}
		$\nu$                             &:~Basis Vectors in $\{e_{P}\}_{r}$; \\
		\noalign{\smallskip}
		$u$                               &:~A vector belongs to $[0,1]^{n}$; \\
		\noalign{\smallskip}
		$d_{t}$&:~Estimated value of the gradient $f(x_{t}, y_{t})$;\\
		\noalign{\smallskip}
		$D$ &:~A constant parameter (In Assumption 12) given in advance, used to\\
            &~~bound the Euclidean norm of a vector $x\in P$;\\
		\noalign{\smallskip}
		$\theta$ &:~A constant parameter given in advance, which is used to\\
                 &~~bound $\nabla f(x,y)$;\\
		\noalign{\smallskip}
		$L$                               &:~The Lipschitz parameter.\\
		\noalign{\smallskip}\hline\noalign{\smallskip}
		\noalign{\smallskip}\hline
	\end{tabular}
\end{table}


\section{Jump-Start Parallel-Greedy (JSPG) Algorithm for Deterministic Setting} \label{sec3}
This section concentrates on the deterministic setting, in which the monotone normlized $OSS$ function $F(x)$ is determined by the decision variable $x\in P\cap [0,1]^n$.The Jump-Start Parallel-Greedy (JSPG)  algorithm is designed for the deterministic $OSS$ problem.

\subsection{JSPG Algorithm Design}\label{sec3.1}
JSPG algorithm is inspired by the work \cite{GSS2021}, in which Mehrdad et al. proposed a Jump-Start Continuous Greedy (JSCG) algorithm to solve the monotone normalized $OSS$ maximization problem. In each iteration of JSCG algorithm, a solution is dynamically improved by adding a direction vector $\nu$ to current solution $x$ with a fixed and conservative step size $\delta>0$. The key idea of the JSCP algorithm actually originated from the continuous-greedy algorithm, where the value of a submodular set function is gradually increased by adding only one element $x$ per iteration. Such an operation leads to a slow speed and does not work well for big data applications.

Our JSPG algorithm also applies the Frank-Wolfe skill, which approximates the objective function as a linear function,  at each iteration point $x_{k}$.
\bea
F(x)\approx F(x_{k})+(x-x_{k})^{T}\nabla F(x_{k}).\nn
\eea
Hence
\bea
\max_{x\in P} F(x) ~~\approx~~\max F(x_{k})+(x-x_{k})^{T}\nabla F(x_{k}), x\in P.\nn
\eea
Since $F(x_{k})$ and $(x_{k})^{T}\nabla F(x_{k})$ are constants. The above optimization problem is thus equivalent to the following problem
\bea
\max x^{T}\nabla F(x_{k}), x\in P.\nn
\eea

Before describing the detailed layout of JSPG algorithm in Algorithm \ref{alg1}, let us introduce it briefly as follows. Given a threshold $\lambda$, JSPG algorithm firstly selects an initial feasible solution
  $$x\leftarrow \alpha\arg max_{x\in P} \|x\|_{1},~ \alpha\rightarrow 0,~\alpha\in (0,1].$$
Based on this initial feasible solution, JSPG algorithm goes to identify all good directions from the base set $\{e_{P}\}_{r}$ by Frank-Wolfe skill, that is choosing all $\{\nu_{i}\in \{e_{P}\}_{r} \mid\nu_{i}^{T}\nabla F(x)\geq(1-\epsilon)\mu\lambda\}$ (If the base set of the polyhedron constraint does not exist, we can replace it with the coordinates of the solution vector, and the algorithm still works). JSPG algorithm continues to increase along all these directions uniformly by a dynamical increment $\delta$. The value of $\delta$ depends on $\mu, \eta$, where $\mu=\left(\frac{\alpha}{\alpha+1}\right)^{2\sigma}$. At last, the threshold $\lambda$ is updated when the inner loop has no direction to choose.

JSPG algorithm improves JSCG algorithm from two following perspectives. In each iteration,

\begin{itemize}
  \item Instead of increasing the current solution $x$ only along a single best direction in JSCG algorithm, our JSPG algorithm identifies all good directions from the base set by Frank-Wolfe skill, and increases along all of these directions uniformly.
  \item Instead of increasing the current solution $x$ along a direction by a fixed step size in JSCG algorithm, our JSPG algorithm dynamically adjusts the step size $\delta$.
\end{itemize}


\begin{algorithm}[h!]
\caption{JSPG: Jump-Start Parallel-Greedy($F, \lambda, P, \epsilon$)}
\label{alg1}
\textbf{Input}: $OSS$ function $F:[0,1]^{n}\cap P\rightarrow \mathbb{R}_{+}$; $P$: down-closed convex polytope with rank $r$; $\{e_{P}\}_{r}$: the set of standard vector bases of $P$. \\
\textbf{Parameter}: $\lambda$: the upper bound of $OPT$, $\alpha\in(0,1]$, $\alpha,\epsilon,\eta>0, \sigma\geq 0$.\\ 
\textbf{Output}: A fractional solution $x$.
\begin{algorithmic}[1] 
\State  $x\leftarrow \alpha\arg \max_{x\in P} \|x\|_{1}$.
  \State  $\mu=\left(\frac{\alpha}{\alpha+1}\right)^{2\sigma}$.
  \State  $t\leftarrow \alpha$.
\While{$t< 1~and~\lambda\geq e^{-\mu} OPT$, }
   \State  Let $M_{\lambda}=\{\nu_{i}\in \{e_{P}\}_{r}: \nu_{i}^{T}\nabla F(x)\geq(1-\epsilon)\mu\lambda\}$.
   \State  $S_{x}\leftarrow M_{\lambda}$.
   \While{$S_{x}$ is not empty and $t< 1$,}
              \State \textbf{A}. Choose maximal $\delta$ s.t.
                   \State \quad (1)~~ $F\left(x+\frac{\delta}{\mid S_{x}\mid}\sum_{\nu_{i}\in S_{x}}\nu_{i}\right)-F(x)$
                   \State~~~~~~
                   $\geq\mu(1-\epsilon)^{2}\frac{\delta}{\mid S_{x}\mid}\|\sum_{\nu_{i}\in S_{x}}\nu_{i}\|\lambda$,
                   \State \quad (2)~~ $\delta\leq \min\{\frac{1}{n\eta}, \frac{1}{\mu(1-\epsilon)}\}$,
                   \State \quad (3)~~ $t+\delta\leq1$.
              \State \textbf{B}. $x\leftarrow x+\frac{\delta}{\mid S_{x}\mid}\sum_{\nu_{i}\in S_{x}}\nu_{i},$ \State~~~~$t\leftarrow t+\delta$.
              \State \textbf{C}. Update
              \State~~~~$S_{x}=\{\nu_{i}\in \{e_{P}\}_{r}: \nu_{i}^{T}\nabla F(x)\geq(1-\epsilon)\mu\lambda\}$. \\
              \EndWhile
   \State $\lambda\leftarrow (1-\epsilon)\lambda$.

\EndWhile
\end{algorithmic}
\end{algorithm}

Our algorithm's parallel technique is different from traditional discrete parallel greedy (DPG) algorithm \cite{brs2019,ENV2019}.
Firstly, the DPG algorithm is actually a traditional discrete algorithm.
Our JSPG algorithm is a parallel one based on Frank-Wolfe skill, which can be used to solve the continuous problem. Secondly, the DPG algorithm is only suitable for the discrete submodular optimization problem, as the property of set submodularity ensures that the DPG algorithm easily addresses the relationship among the $OPT$, the iterative function value $F(x)$ and the threshold $\lambda$. This relationship is critical to obtain the approximation ratio. Our parallel algorithm is not restricted to the property of submodularity and multilinearity, and thus can be applied to all other non-submodular $OSS$ problems. To overcome the missing of submodularity and multilinearity, we use the quadratic differentiability and the $OSS$ property of objective function to analyze JSPG algorithm.

\subsection{Approximation Ratio of JSPG Algorithm} \label{sec3.2}
The property of $OSS$ enables us to bound the first order Taylor's polynomial of the objective function, which is helpful to compute the approximation ratio of JSPG algorithm.

\begin{lemma}\label{lem2}
\cite{GSS2021} Let $x\in [0,1]^{n}\backslash \{0\}, u\in[0,1]^{n}$. There exits $\epsilon >0$ such that $(x+\epsilon u)\in[0,1]^{n}$. Let $F: [0,1]^{n} \rightarrow \mathbb{R}$ be a monotone normalized function which is $OSS$ on $[x,x+\epsilon u]$, then we have
\bea
u^{T}\nabla F(x+\epsilon u)\leq \left(\frac{\|x+\epsilon u\|_{1}}{\|x\|_{1}} \right)^{2\sigma}u^{T}\nabla F(x).\label{eqinlem2}
\eea
\end{lemma}

Because
\bea
\frac{\|x+\epsilon u\|_{1}}{\|x\|_{1}}\leq \frac{\|x+u\|_{1}}{\|x\|_{1}}=1+\frac{\|u\|_{1}}{\|x\|_{1}}\leq 1+\frac{1}{\alpha},\nn
\eea
we continue to explore (\ref{eqinlem2}) to be
\bea
u^{T}\nabla F(x+\epsilon u) \leq \left( \frac{\alpha+1}{\alpha}\right)^{2\sigma }(u^{T} \nabla F(x)).
\eea

\begin{lemma}\label{lem3}
For any solution $x$ obtained in the intermediate stage of JSPG algorithm, we have $OPT-F(x)\leq \lambda$.
\end{lemma}
\begin{proof}
JSPG algorithm starts from a non-zero solution, such that the initial value of objective function $F$ is greater than zero. Moreover, due to the monotone property, $F(x)$ increases, and thus $OPT-F(x)$ monotonically decreases with $x$ increasing. Suppose $z$ is an optimal solution, meaning $OPT=F(z)$. Then $z\vee x\geq z$, ensuring that $F(z\vee x)\geq OPT$. It's a contradiction, and thus $z\vee x\not\in P\cap [0,1]^n$. Let $u=x\vee z-x$, then
\begin{eqnarray}
OPT-F(x)&\leq& F(x\vee z)-F(x)=^{(i)}\langle \nabla F(x+\epsilon u),u\rangle\nonumber\\
 &\leq&^{(ii)} \langle\mu^{-1}\nabla F(x),u\rangle \leq^{(iii)}\langle\mu^{-1}\nabla F(x),z\rangle\leq^{(iv)} \lambda,
\end{eqnarray}
where $\mu=\left(\frac{\alpha}{\alpha+1}\right)^{2\sigma}$. Specifically, equality (i) is from the $Median~Theorem$ of continous function $F$. Lemma 5 guarantees inequality (ii). Combining the monotonicity of $F$, which implies $\nabla F(x)\geq 0$, and the fact of $u\leq z$, inequality (iii) is achieved. The last inequality (iv) is correct, since $z^{T}\nabla F(x)\leq(1-\epsilon)\mu\lambda$. This lemma holds.
\end{proof}

\begin{lemma}\label{lem4}
For any solution $x$ obtained in the intermediate stage of JSPG algorithm, we have
\bea
\frac{1}{\mid S_{x}\mid}\sum_{\nu_{i}\in S_{x}}\nu_{i}^{T}\nabla F(x)
\geq (1-\epsilon)\left(\frac{\alpha}{\alpha+1}\right)^{2\sigma}(OPT-F(x)).
\eea
\end{lemma}
\begin{proof}
By the choices of $\delta$ and $\mu$ in JSPG algorithm, we have
\bea
\nu_{i}^{T}\nabla F(x)\geq(1-\epsilon)\mu \lambda= (1-\epsilon)\left(\frac{\alpha}{\alpha+1}\right)^{2\sigma}\lambda.\label{eqlem4}
\eea
Combining Lemma \ref{lem3} and (\ref{eqlem4}), it is easy to deduce
\bea
\frac{1}{\mid S_{x}\mid}\sum_{\nu_{i}\in S_{x}}\nu_{i}^{T}\nabla F(x)
\geq (1-\epsilon)\left(\frac{\alpha}{\alpha+1}\right)^{2\sigma}(OPT-F(x)).\nn
\eea
This lemma holds.    	
\end{proof}

\begin{theorem}\label{thm1}
When JSPG algorithm terminates, the output solution $x$ satisfies $F(x)\geq(1-O(\epsilon))(1-e^{-\mu})OPT.$
\end{theorem}

\begin{proof}
By Lemma \ref{lem3} and Line 9-10 in Algorithm \ref{alg1}, we have	
\bea
&~&F\left(x+\frac{\delta}{\mid S_{x}\mid}\sum_{\nu_{i}\in S_{x}}\nu_{i}\right)-F(x)\nn\\
&\geq&\mu(1-\epsilon)^{2}\frac{\delta}{\mid S_{x}\mid}\|\sum_{\nu_{i}\in S_{x}}\nu_{i}\|\lambda\label{eqnth8}\\
&\geq&\mu(1-\epsilon)^{2}\frac{\delta}{\mid S_{x}\mid}\|\sum_{\nu_{i}\in S_{x}}\nu_{i}\| \left(OPT-F(x)\right),\nn
\eea

By JSPG algorithm, solution $x$ increases during the whole operation, which leads to the increasement of $F(x)$ by the monotone property of $F(x)$. Thus we analyze the approximation ratio of JSPG algorithm by reducing the dimensionality of variables. Given a solution $x$, let $l_{x}=\sum x_{i}$ be the sum of all elements of $x$. Define a function $G: \mathbb{R}\rightarrow \mathbb{R}$, such that for any $x$, $G(l_x)=F(x)$. So by (\ref{eqnth8}), we have
\bea
G\left(l_{x}+\frac{\delta}{\mid S_{x}\mid}\|\sum_{\nu_{i}\in S_{x}}\nu_{i}\|\right)-G(l_{x})
\geq\mu(1-\epsilon)^{2}\frac{\delta}{\mid S_{x}\mid}\|\sum_{\nu_{i}\in S_{x}}\nu_{i}\| \left(OPT-G(l_{x})\right).\nn
\eea

Based on the choice of $\delta$ in the Algorithm \ref{alg1}, we have
\bea
\frac{dG(l_x)}{dl_x}&=&\lim_{\delta\rightarrow 0}\frac{G\left(l_{x}+\frac{\delta}{\mid S_{x}\mid}\|\sum_{\nu_{i}\in S_{x}}\nu_{i}\|\right)-G(l_{x})}{\frac{\delta}{\mid S_{x}\mid}\|\sum_{\nu_{i}\in S_{x}}\nu_{i}\|}\nn\\
&\geq& \mu(1-\epsilon)^{2}\left(OPT-G(l_{x})\right).
\eea

By Line 7-16 in Algorthm \ref{alg1}, we know the intermediate solution $x$ depends on parameter $t$, thus it can be denoted by $x(t)$. Clearly, because the initial value of $t$ is $\alpha$,  the initial solution is $x(\alpha)$. As the setp size $\delta$ is adjusted iteratively, we denote the step size as $\delta_i$ in the $i$-th iteration for convenience. By the update of $t$ in Line 12 of Algorithm 1, the final value of $t$, denoted by $\hat{t}$, is $\hat{t}=\sum_i \delta_i+\alpha\geq 1$, and thus the final output of Algorthm \ref{alg1} is $x(\hat{t})$.

Next, let us consider the differential equation:
\bea
&~&\frac{d}{dl_x}\left[e^{\mu(1-\epsilon)^{2} l_x} G(l_x)\right]\nn\\
&=&\mu(1-\epsilon)^{2}e^{\mu(1-\epsilon)^{2} l_x} G(l_x)
+e^{\mu(1-\epsilon)^{2} l_x}\frac{d}{dl_x}G(l_x)\nn\\
 &\geq&\mu(1-\epsilon)^{2}e^{\mu(1-\epsilon)^{2} l_x} G(l_x)
 +\mu(1-\epsilon)^{2}e^{\mu(1-\epsilon)^{2} l_x}\left(OPT-G(l_x)\right)\nn\\
 &\geq&\mu(1-\epsilon)^{2}e^{\mu(1-\epsilon)^{2} l_x} OPT.\label{de}
\eea

Intergrating the LHS and RHS of (\ref{de}) between $l_{x(\alpha)}$ and $l_{x(t)}$, we get
\bea
&~&e^{\mu(1-\epsilon)^{2} l_{x(t)}}\cdot G\left(l_{x(t)}\right)-e^{l_{x(\alpha)}}\cdot G\left(l_{x(\alpha)}\right)\nn\\
&\geq&\int_{l_{x(\alpha)}}^{l_{x(t)}}(\mu(1-\epsilon)^{2}e^{\mu(1-\epsilon)^{2} l_{x}}\cdot OPT)dl_x\nn\\
&=&\mu(1-\epsilon)^{2}\cdot OPT\cdot\left[\frac{e^{(1-\epsilon)^{2}\mu l_{x(t)}}-e^{(1-\epsilon)^{2}\mu l_{x(\alpha)}}}{\mu(1-\epsilon)^{2}}\right]\nn\\
&=&OPT\cdot\left[e^{(1-\epsilon)^{2}\mu l_{x(t)}}- e^{(1-\epsilon)^{2}\mu l_{x(\alpha)}}\right].
\eea

Due to $\alpha\rightarrow 0$, then $l_{x(\alpha)}\rightarrow 0$. So
\bea
G\left(l_{x(t)}\right)
&\geq& OPT\cdot\left[1-e^{-(1-\epsilon)^{2}\mu l_{x(t)}}\right]+G\left(l_{x(\alpha)}\right)\cdot e^{-(1-\epsilon)^{2}\mu l_{x(t)}}\nn\\
&\geq& OPT\cdot\left[1-e^{-(1-\epsilon)^{2}\mu l_{x(t)}}\right].\label{deresult}
\eea

According to the termination conditions of JSPG algorithm, following two cases shall be distinguished.\\
\noindent {\bf Case 1.} If $t=\hat{t}\geq 1$ at the end of the algorithm, then $l_{x(\hat{t})}=\sum x_{i}(\hat{t})\geq\|x_{i}(\hat{t})\|_{2} =t$. Hence
\bea
F\left(x(\hat{t})\right)=G\left(l_{x(\hat{t})}\right)\geq (1-O(\epsilon))\left(1-e^{-\mu }\right)\cdot OPT.
\label{result1}
\eea
\noindent {\bf Case 2.}   If $\lambda\leq e^{-\mu}OPT$, then
\bea
F\left(x\right)=G(l_{x})\geq \left(1-e^{-\mu }\right)\cdot OPT.
\label{result2}
\eea
Combining (\ref{result1}) and (\ref{result2}) as above, we eventually have
\bea
F\left(x\right)
\geq(1-O(\epsilon))\left(1-e^{-\left(\frac{\alpha}{\alpha+1}\right)^{2\sigma}}\right)\cdot OPT\nn
\eea

This theorem is obtained.
\end{proof}

\subsection{Number of Iterations $\&$ Number of Oracle Queries}\label{sec3.3}
According to JSPG algorithm, one of $\textbf{A}.(1)$, $\textbf{A}.(2)$, and \textbf{A}.(3) (in Algorithm \ref{alg1}) must tightly hold along the update of $\delta$.
Furthermore, when $\textbf{A}.(3)$ is tight, the algorithm terminates. So we only need to analyze $\textbf{A}.(1)$ and $\textbf{A}.(2)$.

\begin{theorem}\label{thm2} For JSPG algorithm,
if the step size $\delta$ is determined by $\textbf{A}.(2)$ in each iteration, then the inner loop iterates at most $O(n)$ times, and the total loop iterates at most $O(n/\epsilon)$ times. If the step size $\delta$ is determined by $\textbf{A}.(1)$ in each iteration, then the inner loop iterates at most $O(\log n/\epsilon)$ times, and the number of iterations of total loop is at most $O(\log n/\epsilon^{2})$.
\end{theorem}
\begin{proof}
Assuming the step size $\delta$ is determined by $\textbf{A}.(2)$ in each iteration, the result on the number of iterations is not hard to obtain. Next, we need to analyze the case for $\textbf{A}.(1)$.

Define $\triangle=\frac{\delta}{|S|}\|\sum_{\nu_{i}\in S}\nu_{i}\|$. Given the current step size $\delta$. Let $\triangle'$ denote the value of $\triangle$ before updating current $\delta$ and $\triangle''$ denote the value of $\triangle$ after updating. Recall that $l_{x}=\sum x_{i}$ for $x$ and $G$ is a function
$G: \mathbb{R}\rightarrow \mathbb{R}$, satisfying $G(l_x)=F(x)$ for any $x$.
Hence
\bea
\mid\lambda\mu(1-\epsilon)^{2}\triangle'\mid&=&\mid G(x+\triangle')-G(x)\mid=^{{(i)}}\mid\langle\nabla G(x+\epsilon\triangle'),\triangle'\rangle\mid\nn\\
&\geq^{{(ii)}}&\mid\mu^{-1}\langle G'(x+\triangle'),\triangle'\rangle\mid\geq^{{(iii)}}\mid\mu^{-1}\langle G'(x+\triangle'),\triangle''\rangle\mid \nn\\
&\geq^{{(iv)}}&\mid\mu^{-1}(1-\eta\epsilon)\langle G'(x+\triangle''),\triangle''\rangle\mid\label{eqnth09}\\
&\geq&\mid\mu^{-1}(1-\eta\epsilon)\cdot\triangle''\cdot\langle G'(x+\triangle''),\triangle''\rangle\mid\nn\\
&\geq^{{(v)}}&\mid\mu^{-1}\lambda\mu(1-\epsilon)\triangle''\mid\geq\mid\lambda\mu(1-\epsilon)\triangle''\mid.\nn
\eea

Equality $(i)$ is from the $Median~Theorem$ of continous function $F$. Inequality $(ii)$ holds, due to the $OSS$ smoothness and the monotonic decreasing property of $\triangle$. The inequality $(iii)$ is correct since $\delta$ decreases during the whole operation of JSPG algorithm and $\|\frac{\sum_{\nu_{i}\in S_{x}}\nu_{i}}{\mid S_{x}\mid}\|=1$.
Since $F(x)$ is $\eta$-$local$, we get inequality $(iv)$.   The inequality $(v)$ holds, due to the choice of $S_{x}$. So from (\ref{eqnth09}), we have $\triangle''\leq (1-\epsilon)\triangle'$. Noth that $\triangle'-\triangle''$ is just the increment of $x$ between two adjacent iterations. Let $k$ be the total number of iterations. Based on the
the choice of $\delta$ (suppose the value of $\delta$ is $O(n^{-c})$, $c$ represents any finite number) and due to the fact of $\max_{x\in P\cap [0,1]^{n}}\|x\|\leq r$ (where $r$ is the rank of $P$), $k$ can be computed through $r(1-\epsilon)^{k}=n^{-c}$. So the number of the iterations of inner is at most $O\left(\frac{\log n}{\epsilon}\right).$

In addition, $\lambda$ is the upper bound of $OPT$ at the beginning of JSPG algorithm. It updates as $\lambda\leftarrow(1-\epsilon)\lambda$ during the operation. So one of terminate conditions $\lambda< e^{-\mu} OPT$ makes us
get that the outer loop iterates at most $O\left(\frac{1}{\epsilon}\right)$ times, through the equation $OPT(1-\epsilon)^{k}=e^{-\mu} OPT$.

Therefore, we conclude that the total loop iterate at most
$ O\left(\frac{\log n}{\epsilon^{2}}\right)$ times. This theorem holds.
\end{proof}

By Algorithm \ref{alg1}, the number of oracle queries to $F$ is at most $O(\log n/\epsilon^{2})$, and the number of oracle queries to $\nabla F(x)$ is at most $O(n \log n/\epsilon^{2})$. Thus, we obtain the following result.

\begin{theorem}\label{thm7}
The number of oracle queries to $F$ is at most $O(\log n/\epsilon^{2})$, and the number of oracle queries to $\nabla F(x)$ is at most $O(n \log n/\epsilon^{2})$.
\end{theorem}

Traditionally continuous greedy algorithms mainly appear in solving multilinear extension problems of the submodular set function. The key idea is that the partial derivative of the function is equal to the marginal return of the corresponding component, which is guaranteed by the property of submodularity. To avoid using the property of submodularity, our algorithm applies the idea of Frank-Wolfe Algorithm with the help of the quadratic differentiability and the $OSS$ property of objective function. This makes our algorithm be applied more widely. Compared with the algorithm in \cite{GSS2021}, our algorithm can ensure the complexity of the sub-optimization problem $\max_{\nu\in P}\nu^{T}\nabla F(x)$ to be quantifiable. The $OSS$ problems are mainly continuous problems, which are also applicable to the continuous expansion of some combination problems. The traditional rounding techniques mentioned in the paper \cite{GSS2021} are all available.

\section{Stochastic Parallel-Greedy (SPG) Algorithm for Stochastic Setting} \label{sec4}
In this section, we continue to study the problem of maximizing a monotone normalized $OSS$ function $F(x)$ under a stochastic setting. In this stochastic version, the objective function $F(x)$ is defined as $F(x)=\mathbb{E}_{y\sim T}f(x,y)$, where $f$ is a stochastic function with respect to the random variable $Y$, and $y$ is the realization of $Y$ drawn from a probability distribution $T$. In this section, we design a Stochastic Parallel-Greedy (SPG for short) algorithm for the stochastic $OSS$ maximization problem.

\subsection{Stochastic Parallel-Greedy (SPG) Algorithm Design}\label{sec4.1}
Stochastic Pa-rallel-Greedy algorithm is a stochastic variant of JSPG algorithm. It starts from zero instead of jumping start. For JSPG algorithm, we need to accurately calculate the gradient value. But it is difficult to get the exact value of $\nabla F(x)$ if the objective function obeys a unknown or complex probability distribution. To overcome this difficulty, we use the estimated gradient $$d_{t}=(1-\rho_{t})d_{t-1}+\rho_{t}\nabla f(x_{t}, y_{t}),$$
where $\rho_{t}=(\frac{4}{t+8})^{\frac23}$ is a positive step size, dependent on $t$ and the initial vector $d_{0}$ is $\mathbf{0}$.

To design a stochastic parallel algorithm, we need to find a good alternative to the gradient function, but it is not enough to rely on $d_{t}$. So three other assumptions are required:
\begin{assumption}\label{assump4}
The Euclidean norm of the elements in the down-closed convex polytope $P$ are uniformly bounded, i.e., for all $x\in P$, $\|x\|\leq D.$
\end{assumption}

\begin{assumption}\label{assump5}
The gradients of objective function $F(x)$
are $L$-$\emph{Lipschitz}$ continuous over the sets $X$. That is for all $x,~x'\in X$, we have
\bea
\|\nabla F(x)-\nabla F(x')\|\leq L\|x-x'\|.
\eea
\end{assumption}

\begin{assumption}\label{assump6}
The variance of the unbiased stochastic gradients $\nabla f(x,y)$ is bounded above by $\theta^{2}$. That is for all $x\in X$, we have
\bea
\mathbb{E}\|\nabla F(x,y)-\nabla f(x,y)\|\leq \theta^{2}.
\eea
\end{assumption}

Under Assumptions \ref{assump4}-\ref{assump6}, $E[\|\nabla F(x_{t})-d_{t}\|^{2}]$ is bounded by a parameter \cite{mhk2020}. This conclusion plays an important role in our design of parallel algorithms for solving stochastic $OSS$ problems.

Our SPG algorithm in Algorithm. \ref{alg2} includes two phases. The first phase is to find all good directions by solve $\{\nu_{i}\in \{e_{P}\}_{r}: \nu_{i}^{T}d_{t}\geq(1-\epsilon)\mu\lambda\}$, and to increase along all of these directions uniformly in Line 6-7; The second phase is to increase $x$ along these directions by a dynamical increment $\delta$ in Line 8-18.

\begin{algorithm}[h!]
\caption{SPG: stochastic Parallel-Greedy($F, \lambda, P, \epsilon$)}
\label{alg2}
\textbf{Input}: $OSS$ function $F(x):=\mathbb{E}_{y\sim T }\left[f(x,y)\right]:X\times Y\rightarrow R_{+}$; $P$: down-closed convex polytope of rank $r$; $\{e_{P}\}_{r}$: the set of standard vector bases of $P$, $\textbf{0}\in P$. \\
\textbf{Parameter}: $\lambda$: the upper bound of $OPT$, $\epsilon,\eta,\sigma,\kappa,\rho_{t}\geq 0$, $\alpha \in(0,1]$.\\
\textbf{Output}: A fractional solution $x$.
\begin{algorithmic}[1] 
\State  $x\leftarrow \textbf{0}$.
  \State  $\mu=\left(\frac{\alpha}{\alpha+1}\right)^{2\sigma}$.
  \State  $t\leftarrow 0$.
  \State  $d_{t}=\textbf{0}$.
\While{$t< 1~and~\lambda\geq e^{-\mu} OPT$, }
   \State  Let $M_{\lambda}=\{\nu_{i}\in \{e_{P}\}_{r}: \nu_{i}^{T}d_{t}\geq(1-\epsilon)\mu\lambda\}$.
   \State  $S_{x}\leftarrow M_{\lambda}$.
   \While{$S_{x}$ is not empty and $t< 1$,}
              \State \textbf{A}. Choose $\delta$ maximal s.t.
                   \State \quad 1)~~ $F\left(x+\frac{\delta}{\mid S_{x}\mid}\sum_{\nu_{i}\in S_{x}}\nu_{i}\right)-F(x)$
                   \State~~~~~~ $\geq\mu(1-\epsilon)^{2}\frac{\delta}{\mid S_{x}\mid}\|\sum_{\nu_{i}\in S_{x}}\nu_{i}\|(\lambda+\mu^{-1}\kappa^{1/2}r)$,
                   \State \quad 2)~~ $\delta\leq \min\{\frac{1}{n\eta}, \frac{1}{\mu^{2}(1-\epsilon)}\}$,
                   \State \quad 3)~~ $t+\delta\leq 1$.
              \State \textbf{B}. $x\leftarrow x+\frac{\delta}{\mid S_{x}\mid}\sum_{\nu_{i}\in S_{x}}\nu_{i}$,
              \State ~~~~~$d_{t}=(1-\rho_{t})d_{t'}+\rho_{t}\nabla f(x,y)$,
              \State ~~~~~$t\leftarrow t+\delta$. \\
              (Where $t'$ represents the moment before $t$ is updated, $\rho_{t}=\frac{4}{t+8}^{2/3}$).
              \State \textbf{C}. Update \\
              \State~~~~$S_{x}=\{\nu_{i}\in \{e_{P}\}_{r}: \nu_{i}^{T}d_{t}\geq(1-\epsilon)\mu\lambda\}.$
              \EndWhile
   \State $\lambda\leftarrow (1-\epsilon)\lambda$.

\EndWhile
\end{algorithmic}
\end{algorithm}

\subsection{Approximation Ratio of SPG Algorithm}\label{sec4.2}

\begin{lemma}\label{lem5}
\cite{mhk2020} Under Assumptions \ref{assump4}-\ref{assump6}, we have
\bea
\mathbb{E}[\|\nabla F(x_{t})-d_{t}\|^{2}]\leq \kappa,
\eea
where $\kappa=\frac{\max \{5\|\nabla F(x_{0})-d_{o}\|^{2}, 16\theta^{2}+2L^{2}D^{2}\}}{(t+9)^{2/3}}$.
\end{lemma}

\begin{lemma}\label{lem6}
For any solution $x$ obtained in the intermediate stage of SPG algorithm,
we have
\bea
OPT-F(x)\leq \lambda+\mu^{-1}\kappa^{1/2}r.
\eea
\end{lemma}

\begin{proof}
During the operation of SPG algorithm, the value of $F(x)$ starts from a number greater than zero and increases monotonically. Thus $OPT-F(x)$ gradually decreases with $x$ increasing. Let $z$ be an optimal solution, then
\bea
&~&OPT-F(x)\leq F(x\vee z)-F(x)=\langle \nabla F(x+\epsilon u),u\rangle\nn\\
&\leq&\langle\mu^{-1}\nabla F(x),u\rangle\leq\langle\mu^{-1}\nabla F(x),z\rangle\leq \mu^{-1}\langle d_{t}+(\nabla F(x)-d_{t}),z\rangle\\
&\leq& \lambda+\mu^{-1}\kappa^{1/2}r.\nn
\eea

The analysis is similar to the proof for Lemma \ref{lem3}, and so we don't explain it in detail.
\end{proof}

\begin{theorem}\label{thm4}
Let $F$ be a $\eta$-local and monotone normalized $OSS$ function and $P$ be a convex polytope including $\textbf{0}$. By setting $\delta\leq\min\{\frac{1}{\eta n}, \frac{1}{(1-\epsilon)^{2}\mu}\}$, the value of the output $F(x)$ from SPG algorithm is larger than
$$F(x)\geq(1-O(\epsilon)\left(1-e^{-\left(\frac{\alpha}{\alpha+1}\right)^{2\sigma}}\right)(1-o(1))OPT-O(\kappa^{1/2})$$ for any given parameter $\alpha\in(0,1]$, $\sigma,\epsilon,\kappa$.
\end{theorem}
\begin{proof}
Let $x(t)$ represents the variable $x$ related to time $t$, and $x(0)$ and $x(1)$ are used to denote the initial and the final solutions, respectively. Since $F(x)$ is twice continuous differentiable and $P$ is a convex set,
\bea
F(x(t+\delta))&=&F\left(x(t)+\frac{\delta}{\mid S_{x(t)}\mid}\sum_{\nu_{i}\in S_{x(t)}}\nu_{i}\right)\nn\\
&=&F(x(t))+\frac{\delta}{\mid S_{x(t)}\mid}\sum_{\nu_{i}\in S_{x(t)}}\nu_{i}\cdot\nabla F\left(x(t)+\frac{\epsilon\delta}{\mid S_{x(t)}\mid}\sum_{\nu_{i}\in S_{x(t)}}\nu_{i}\right),\nn
\eea
where
\bea
&~&\frac{\delta}{\mid S_{x(t)}\mid}\sum_{\nu_{i}\in S_{x(t)}}\nu_{i}\cdot\nabla F\left(x(t)+\frac{\epsilon\delta}{\mid S_{x(t)}\mid}\sum_{\nu_{i}\in S_{x(t)}}\nu_{i}\right)\nn\\
 &\geq& (1-\epsilon\delta\eta)\frac{\delta}{\mid S_{x(t)}\mid}\sum_{\nu_{i}\in S_{x(t)}}\nu_{i}\cdot\nabla F\left(x(t)\right)\nn\\
 &\geq& (1-\epsilon\delta\eta)\frac{\delta}{\mid S_{x(t)}\mid}\sum_{\nu_{i}\in S_{x(t)}}\nu_{i}\cdot\left(d_{t}+(\nabla F\left(x(t)-d_{t})\right)\right)\nn\\
 &\geq& (1-\epsilon\delta\eta)\frac{\delta}{\mid S_{x(t)}\mid}\sum_{\nu_{i}\in S_{x(t)}}\nu_{i}\cdot d_{t}
  +(1-\epsilon\delta\eta)\frac{\delta}{\mid S_{x(t)}\mid}\sum_{\nu_{i}\in S_{x(t)}}\nu_{i}\cdot\left(\nabla F\left(x(t)\right)-d_{t}\right)\nn\\
 &\geq& (1-\epsilon\delta\eta)\cdot\delta(1-\epsilon)\mu\lambda
 +(1-\epsilon\delta\eta)\frac{\delta}{\mid S_{x(t)}\mid}\sum_{\nu_{i}\in S_{x(t)}}\nu_{i}\cdot\left(\nabla F\left(x(t)\right)-d_{t}\right)\nn\\
 &\geq&  (1-\epsilon\delta\eta)\cdot(1-\epsilon)\delta\mu\lambda
 -(1-\epsilon\delta\eta)\delta\|\nabla F\left(x(t)\right)-d_{t}\|\nn\\
 &\geq&  (1-\epsilon\delta\eta)\cdot(1-\epsilon)\delta\mu(OPT-F(x(t))-\mu^{-1}\kappa^{1/2}r)
 -(1-\epsilon\delta\eta)\delta\|\nabla F\left(x(t)\right)-d_{t}\|\nn\\
 &\geq&  (1-\epsilon\delta\eta)\cdot(1-\epsilon)\delta\mu(OPT-F(x(t))
 -(1-\epsilon\delta\eta)\kappa^{1/2}(\mu r+1).\nn
\eea

Define $M=(1-\epsilon\delta\eta)OPT$. Because $F$ is nonnegative, $\delta\leq\frac{1}{\eta n}$, and $\|\nabla F\left(x(t)\right)-d_{t}\|\leq \kappa^{1/2}$, we have
\bea
F(x(t+\delta))\geq F(x(t))+(1-\epsilon)\mu\delta(M-F(x(t)))-(1-\epsilon\delta\eta)\kappa^{1/2}(\mu r+1),\nn
\eea
which can be deduced as
\bea
(1-(1-\epsilon)\mu\delta)(M-F(x(t)))\geq(M-F(x(t+\delta)))-(1-\epsilon\delta\eta)\kappa^{1/2}(\mu r+1).\nn
\eea

By induction,
\bea
(1-(1-\epsilon)\mu\delta)^{1/\delta}(M-F(x(0)))\geq (M-F(x(1)))-(1-\epsilon\delta\eta)\kappa^{1/2}(\mu r+1).\nn
\eea
Since $\delta\leq \frac{1}{\mu(1-\epsilon)}$, we have
\bea
\left(1-(1-\epsilon)\mu\delta\right)^{1/\delta}\leq e^{-\mu(1-\epsilon)}.\nn
\eea
Therefore,
\bea
F(x(1))
&\geq&(1-e^{-\mu(1-\epsilon)})M+F(x(0)-(1-\epsilon\delta\eta)\kappa^{1/2}(\mu r+1)\nn\\
 &\geq&(1-e^{-\mu(1-\epsilon)})M-(1-\epsilon\delta\eta)\delta\kappa^{1/2}(\mu r+1) \nn\\
 &=&(1-e^{-\mu(1-\epsilon)})(1-\epsilon\delta\eta)OPT-(1-\epsilon\delta\eta)\kappa^{1/2}(\mu r+1) \nn\\
 &=&(1-O(\epsilon)[1-e^{-\mu}](1-o(1))OPT-O(\delta\kappa^{1/2}).  \nn
\eea
This theorem holds.
\end{proof}

\section{Conclusion}\label{sec5}

In this paper, we study the problem of maximizing a monotone normalized $OSS$ function $F(x)$, subject to a convex polytope constraint. The deterministic and stochastic versions of this problems are both considered and we respectively design two parallel algorithms for them. Specifically, for the deterministic $OSS$ maximization problem, we propose a $((1-e^{-\left(\frac{\alpha}{\alpha+1}\right)^{2\sigma}})-\epsilon)$-approximation JSPG algorithm for any any number $\alpha\in[0,1]$ and $\epsilon>0$. The time complexity of this deterministic algorithm is $O(\log n/\epsilon^{2})$, in while, the number of oracle queries to $F$ is at most $O(\log n/\epsilon^{2})$, and the number of oracle queries to $\nabla F(x)$ at most $O(n \log n/\epsilon^{2})$. For the stochastic $OSS$ maximization problem, the designed SPG algorithm outputs a result of $(F(x)\geq(1 -e^{-\left(\frac{\alpha}{\alpha+1}\right)^{2\sigma}}-\epsilon)OPT-O(\kappa^{1/2}))$ ($\kappa=\frac{\max \{5\|\nabla F(x_{0})-d_{o}\|^{2}, 16\sigma^{2}+2L^{2}D^{2}\}}{(t+9)^{2/3}}$) under the same time complexity with the one of JSPG algorithm.

The above two algorithms are the first parallel algorithms proposed to solve monotone normalized $OSS$ problem subject to a convex polytope constraint (no need to downward-closed \cite{GSS2021}). Most of $OSS$ problems are continuous, and thus our algorithms can be applied directly. However, if the optimization problems are discrete, then we need to round the corresponding non-discrete results. Generally, the traditional rounding techniques mentioned in \cite{GSS2021,cq2019} are all available. JSPG and SPG algorithms are also suitable for nonmonotone situations, but the performance of them under nonmonotone situations may be poor. Therefore, designing an efficient parallel algorithm for nonmonotone $OSS$ problems is an interesting topic.

\bmhead{Acknowledgments}
The first and the fourth authors are supported by Beijing Natural Science Foundation Project No. Z200002 and National Natural Science Foundation of China (No. 12131003). The second author is supported from the National Natural Science Foundation of China (No. 11871366) and Qing Lan Project of Jiangsu Province, China.

\begin{appendices}

\section{Examples of the $OSS$ function}\label{secA1}

In this section, we propose two examples of the one-sided $\sigma$-smooth function.

\begin{itemize}
  \item {\bf $\sigma=0$ : Continuous $DR$-submodular function (e.g. Multilinear extension of set functions \cite{cq2019})}\\
  \begin{proof}
  Firstly, we give the definition of the continuous $DR$-submodular function
  \begin{definition}\label{def7}
A continuously twice differentiable function $F: \mathbb{R}^{n}_{\geq 0} \rightarrow \mathbb{R}$ is $DR$-submodular if it satisfies
\bea
F(ke_{i}+x)-F(x)\geq F((k+l)e_{i}+x)-F(le_{i}+x),\label{eqde11}
\eea
where  $k,l\in \mathbb{R}_{+}$ and $x,(ke_{i}+x),((k+l)e_{i}+x)\in\mathbb{R}^{n}_{\geq 0}$.\nn
\end{definition}

In \cite{abk2006}, Bian et.al. proposed the second-order condition of continuous $DR$-submodular function: a continuously twice differentiable function $F: \mathbb{R}^{n}_{\geq 0} \rightarrow \mathbb{R}$ is $DR$-submodular if and only if
\bea
\frac{\partial^{2} F}{\partial x_{i}\partial x_{j}}\leq 0,~~\forall i,j\in [n].\nn
\eea

Next, let $A$ denote the second-order Hessian matrix of the $DR$-submodular function, we get $A_{ij}\leq0$ for any $i,j\in [n]$. then for any vector $u=(u_{1},...,u_{n})^{T}\geq 0$, the following inequality holds
\bea
u^{T}Au&=&u_{1}^{2}A_{11}+,...,+u_{1}u_{n}A_{n1}\nn\\
&~&+u_{1}^{2}A_{12}+,...,+u_{1}u_{n}A_{n2}\nn\\
&~&+,...,\nn\\
&~&+u_{1}^{2}A_{1n}+,...,+u_{1}u_{n}A_{nn}\nn\\
&\leq& 0.\nn
\eea

So the continuous $DR$-submodular function is one-sided $0-$smooth

\end{proof}
    \item {\bf $\sigma>0$ :} \cite{GSS2021} $F(x)=\frac{1}{2} x^{T} Mx+ b^{T}x$ is $OSS$ if $M$ is
a $\sigma$-semi-metric. Where $M\in \mathbb{R}^{n\times n}, b\in \mathbb{R}^{n}, b\geq\textbf{0}$

\begin{proof}
Let $M\in \mathbb{R}^{n\times n}$ be a non-negative symmetric $\sigma$-semi-metric ($\sigma$-semi-metric means that $M_{i,j}\leq\sigma (M_{i,k}+M_{k,j})$ holds for any $i,j,k\in [n])$. Note that
$\nabla^{2} F(x)=M, \nabla F(x)=Mx+b$. So
\bea
\sigma(\nabla_{i} F(x)+\nabla_{j} F(x))
&\geq&\sigma(\sum_{k=1}^{n}M_{i,k}x_{k}+\sum_{k=1}^{n}M_{j,k}x_{k})\nn\\
&=&\sum_{k=1}^{n}\sigma(M_{i,k}+M_{j,k})x_{k}
\geq\sum_{k=1}^{n}M_{i,j}x_{k}\nn\\
&=&\|x\|_{1}M_{i,j}=\|x\|_{1}\nabla^{2}_{i,j}F(x)\nn
\eea
where $u=(e_{i}, e_{j})^{T}$.

\end{proof}
\end{itemize}

\section{Discretization of JSPG Algorithm}\label{secA2}
The parallel algorithms are generally applied for the situation of big data. The proof for the approximation ratio in Theorem \ref{thm1} strictly depends on the infinite value of $n$. However, if $n$ is a finite number, then parameter $\delta$ cannot be arbitrarily small, and the process of integration in the analysis may not be realized. So it is necessary for us study the process from the perspective of discretization.
In fact, our JSPG algorithm can run in a polynomial time by the bounded locality and smoothness conditions. In the following, we propose a different analysis, that relies on $OSS$ smoothness and $\eta$-local boundedness, for the discretization version.

\begin{theorem}\label{thm3}
Let $F:[0,1]^{n}\rightarrow \mathbb{R}_{+}$ be a $\eta$-local monotone normalized $OSS$ function and $P$ be a convex polytope. By setting $\delta\leq\min\{\frac{1}{\eta n}, \frac{1}{(1-\epsilon)^{2}\mu}\}$, the approximation ratio of JSPG algorithm is $((1-e^{-\left(\frac{\alpha}{\alpha+1}\right)^{2\sigma}})-\epsilon)(1-o(1))$, for any given parameter $\alpha\in(0,1]$, $\sigma,\epsilon>0$.
\end{theorem}
\begin{proof}
Recall that $x(t)$ is the intermediate solution $x$ corresponding to parameter $t$. So the initial solution is $x(\alpha)$, and final output is $x(\hat{t})$.
Because $F(x)$ is twice continuous differentiable and $P$ is a convex set, we have
\bea
F(x(t+\delta))
 &=&F\left(x(t)+\frac{\delta}{\mid S_{x(t)}\mid}\sum_{\nu_{i}\in S_{x(t)}}\nu_{i}\right)\label{eqnth11}\\
 &=&F(x(t))+\frac{\delta}{\mid S_{x(t)}\mid}\sum_{\nu_{i}\in S_{x(t)}}\nu_{i}\cdot\nabla \nn F\left(x(t)+\frac{\epsilon\delta}{\mid S_{x(t)}\mid}\sum_{\nu_{i}\in S_{x(t)}}\nu_{i}\right).\nn
\eea

The second part of RHS in (\ref{eqnth11}) is further deduced as
\bea
&~&\frac{\delta}{\mid S_{x(t)}\mid}\sum_{\nu_{i}\in S_{x(t)}}\nu_{i}\cdot\nabla F\left(x(t)+\frac{\epsilon\delta}{\mid S_{x(t)}\mid}\sum_{\nu_{i}\in S_{x(t)}}\nu_{i}\right)\nn\\
 &\geq& (1-\epsilon\delta\eta)\frac{\delta}{\mid S_{x(t)}\mid}\sum_{\nu_{i}\in S_{x(t)}}\nu_{i}\cdot\nabla F\left(x(t)\right)\nn\\
 &\geq& (1-\epsilon\delta\eta)\cdot\delta\mu\lambda(1-\epsilon)\\
 &\geq& (1-\epsilon\delta\eta)\cdot\delta\mu(OPT-F(x(t)))(1-\epsilon), \nn
\eea
where the first, the second and the the last inequalities are from the property of $\eta$-local, Lemma \ref{lem4} and Lemma \ref{lem3}, respectively.

Now let us define $M=(1-\epsilon\delta\eta)OPT$. Since $F(x)$ is non-negative and $\delta\leq\frac{1}{\eta n}$,
\bea
F(x(t+\delta))\geq F(x(t))+(1-\epsilon)\mu\delta(M-F(x(t))),\nn
\eea
from which we can deduce that
\bea
(1-(1-\epsilon)\mu\delta)(M-F(x(t)))\geq (M-F(x(t+\delta))).\nn
\eea

Applying the fact of $\hat{t}-\alpha\approx O(1/\delta)$, the following can be obtained by induction
\bea
(1-(1-\epsilon)\mu\delta)^{1/\delta}(M-F(x(\alpha)))\geq (M-F(x(\hat{t}))).\label{eqnth11-2}
\eea

Since $\delta\leq \frac{1}{\mu(1-\epsilon)}$, we have
\bea
\left(1-(1-\epsilon)\mu\delta\right)^{1/\delta}\leq e^{-\mu(1-\epsilon)}.\label{eqnth11-3}
\eea

Substitute (\ref{eqnth11-3}) into (\ref{eqnth11-2}),
\bea
F(x(\hat{t}))&\geq&(1-e^{-\mu(1-\epsilon)})M+F(x(\alpha))\nn\\
&\geq&(1-e^{-\mu(1-\epsilon)})M \nn\\
 &=&(1-e^{-\mu(1-\epsilon)})(1-\epsilon\delta\eta)OPT \nn\\
 &=&(1-O(\epsilon)[1-e^{-\mu}](1-o(1))OPT.  \nn
\eea

This theorem holds.
\end{proof}

\end{appendices}




\end{document}